\documentclass[10pt, oneside]{article}

\usepackage[dvips]{graphicx}
\usepackage{pstricks}

\usepackage{amssymb}
\usepackage{amsmath}
\usepackage{amsthm}
\usepackage{color}

\newtheorem{theorem}{Theorem}[section]

\newtheorem{lemma}[theorem]{Lemma}

\theoremstyle{definition}
\newtheorem{definition}{Definition}[section]

\theoremstyle{remark}
\newtheorem{remark}{Remark}[section]

\newcommand{\RR}{\mathbb{R}}
\newcommand{\NN}{\mathbb{N}}

\begin{document}
\title{Radially bounded solutions of a $k$-Hessian equation involving a weighted nonlinear source \thanks{Partially supported by Grant FONDECYT 1150230}}

\author{Justino S\'anchez\,\,$^{a}$\; and\;
Vicente Vergara \,$^{a,b}$}
\date{}
\maketitle

\begin{center}
$^a$\,\footnotesize{Departamento de Matem\'{a}ticas, Universidad de La Serena\\
 Avenida Cisternas 1200, La Serena 1700000, Chile.}
\\email: jsanchez@userena.cl
\end{center}
\begin{center}
$^b$\,\footnotesize{Universidad de Tarapac\'{a}, Avenida General Vel\'{a}squez 1775,\\
Arica, Chile.}
\\email: vvergaraa@uta.cl
\end{center}

\begin{abstract}
We consider the problem
\begin{equation}\label{Eq:Abstract}
\begin{cases}
S_k(D^2u)= \lambda |x|^{\sigma} (1-u)^q &\mbox{in }\;\; B,\\
u <0 & \mbox{in }\;\; B,\\
u=0 &\mbox{on }\partial B,
\end{cases}
\end{equation}
where $B$ denotes the unit ball in $\mathbb{R}^n$, $n>2k$ ($k\in \mathbb{N}$), $\lambda>0$, $q > k$ and $\sigma\geq 0$. We study the existence, uniqueness and multiplicity of negative bounded radially symmetric solutions of \eqref{Eq:Abstract}. The methodology to obtain our results is based on a dynamical system approach. For this, we introduce a new transformation which reduces problem \eqref{Eq:Abstract} to an autonomous two dimensional generalized Lotka-Volterra system.

\end{abstract}
2010 Mathematics Subject Classification. Primary: ; Secondary: .


\section{Introduction}
Let $k \in \NN$ and let $\Omega$ be a suitable bounded domain in $\RR^n$. We consider the nonlinear problem
\begin{equation}\label{Eq:f}
\begin{cases}
S_k(D^2u)= f(x,u)&\mbox{in }\;\; \Omega,\\
u <0 & \mbox{in }\;\; \Omega,\\
u=0 &\mbox{on }\;\; \partial\Omega,
\end{cases}
\end{equation}
where $S_k(D^2u)$ stands for the $k$-Hessian operator of $u$ and $f$ is a given nonlinear source. Problem \eqref{Eq:f} has been studied extensively by many authors in different settings. See e.g. \cite{CaNS85, ChWa01, Trudinger95, Trudinger97, Tso89, Urbas90, Wang94}.\\
The $k$-Hessian operator $S_k$ is defined as follows. Let $u\in C^2(\Omega)$, $1\leq k\leq n$, and let $\Lambda=(\lambda_1,\lambda_2,...,\lambda_n)$ be the eigenvalues of the Hessian matrix $(D^2u)$. Then the $k$-Hessian operator is given by
\[
S_k(D^2u)=P_k(\Lambda)=\sum_{1\leq i_1<...<i_k\leq n}\lambda_{i_1}...\lambda_{i_k},
\]
where $P_k(\Lambda)$ is the $k$-th elementary symmetric polynomial in the eigenvalues $\Lambda$, see e.g. \cite{Wang94, Wang09}. Note that $\{S_k:k=1,...,n\}$ is a family of operators which contains the Laplace operator ($k=1$) and the Monge-Amp\`{e}re operator ($k=n$). The monograph \cite{CaMi98} is devoted to applications of Monge-Amp\`{e}re equations to geometry and optimization theory. This family of operators has been studied extensively, see e.g. \cite{Jacobsen99, Tso90} and the references therein. Recently, this class of operators has attracted renewed interest, see e.g. \cite{Bran13, Gavitone09, Gavitone10, NaTa15, DeGa14, WaBa13, WaXu14, SaVe15}.\\
We point out that the $k$-Hessian operators are fully nonlinear for $k\neq 1$. Further, they are not elliptic in general, unless they are restricted to the class
\begin{equation}\label{phi-k-0}
\Phi_0^k(\Omega)=\{u\in C^2(\Omega)\cap C(\overline{\Omega}): S_{i}(D^2 u)\geq 0\;\;\mbox{in}\;\;\Omega,\, i=1,...,k,\,u=0\;\; \text{on}\;\; \partial\Omega \}.
\end{equation}
Observe that $\Phi_0^k(\Omega)$ belongs to the class of subharmonic functions. Further, the functions in $\Phi_0^k(\Omega)$ are negative in $\Omega$ by the maximum principle, see \cite{Wang94}. The $k$-Hessian operator defined on $\Phi_0^k(\Omega)$ imposes certain geometry restrictions on $\Omega$. More precisely, domains called admissible are those whose boundary $\partial\Omega$ satisfies the inequality
\begin{equation}\label{curv:1}
P_{k-1}(\kappa_1, ..., \kappa_{n-1})\geq 0,
\end{equation}
where $\kappa_1, ..., \kappa_{n-1}$ denote the principal curvatures of $\partial\Omega$ relative to the interior normal. A typical example of a domain $\Omega$ for which \eqref{curv:1} holds is a ball. For more details we refer the interested reader to \cite{Wang09}.

\begin{remark}
Problem \eqref{Eq:f} can be easily reformulated in order to study positive solutions under the change of variable $v=-u$, which in turn yields $S_k(D^2u)=(-1)^kS_k(D^2v)$ by the $k$-homogeneity of the $k$-Hessian operator.
\end{remark}
\medbreak

Now observe that, if $u\in \Phi_0^k(\Omega)$, then the right hand side of (\ref{Eq:f}) must be nonnegative. Typical examples of nonlinear terms $f$ appearing in the literature are $f(u) = |\lambda u|^p$ (see \cite{Tso90}), $f(u)= \lambda e^{-u}$ (see \cite{Chandrasekhar85, Fowler31, Gelfand63, JoLu73} for $k=1$ and \cite{Jacobsen99, Jacobsen04, JaSc02, JaSc04} for $1\leq k \leq n$) and  $f(u)=\lambda(1+u)^p$ (see \cite{BrVa97, GaMa02, JoLu73} for $k=1$.)
The seminal contribution on the analysis of critical values for $k=1$ with a polynomial and exponential source was made by Joseph and Lundgren in \cite{JoLu73}. In general, for problems of Gelfand type for $k\geq 1$, the first result ($k=n$) in the radial case is due to Cl\'ement et al. \cite{ClFM96} and for $1\leq k \leq n$ to Jacobsen \cite{Jacobsen04}. Note that in all the examples above the nonlinear terms are independent of the variable $x$.

\medbreak

Next, for our purposes we give some general notions of solutions to (\ref{Eq:f}). As usual, a classical solution (or solution) of (\ref{Eq:f}) is a function $u\in \Phi_0^k(\Omega)$ satisfying the equation in (\ref{Eq:f}). We recall the version of the method of super and subsolutions for \eqref{Eq:f}, see \cite[Theorem 3.3]{Wang94} for more details.
\begin{definition}
A function $u\in\Phi^k(\Omega):=\{u\in C^2(\Omega)\cap C(\overline{\Omega}): S_{i}(D^2 u)\geq 0\;\;\mbox{in}\;\;\Omega,\, i=1,...,k,\}$ is called a {\it subsolution} (resp. {\it supersolution}) of \eqref{Eq:f} if
\begin{equation*}
\begin{cases}
S_k(D^2u)\geq (\mbox{resp.}\leq)& f(x,u)\;\;\mbox{in }\;\; \Omega,\\
u\leq (\mbox{resp.}\geq)\;\; 0&\qquad\;\,\mbox{on }\; \partial \Omega.
\end{cases}
\end{equation*}
\end{definition}
Note that the trivial function $u\equiv 0$ is always a supersolution.

The following concept is needed to establish a general result on the existence of solutions to problem \eqref{Eq:f}.
\begin{definition}
We say that a function $v$ is a {\it maximal} solution of \eqref{Eq:f} if $v$ is a solution of \eqref{Eq:f} and, for each subsolution $u$ of \eqref{Eq:f}, we have $u\leq v$.
\end{definition}
This notion of maximal solution was recently introduced in \cite{SaVe15} to prove the existence of solutions to problem \eqref{Eq:f}.

In this article we study problem (\ref{Eq:f}) on the unit ball $B$ of $\RR^n$ involving a weight and a polynomial source, i.e, let us consider the problem
\begin{equation}\label{Eq:f:pol}
\begin{cases}
S_k(D^2u)= \lambda |x|^\sigma (1-u)^q &\mbox{in }\;\; B,\\
u <0 & \mbox{in }\;\; B,\\
u=0 &\mbox{on }\partial B,
\end{cases}
\end{equation}
where $\lambda\in\RR$ is a parameter and $q>0$. In the semilinear case ($k=1$) problem \eqref{Eq:f:pol} was studied via bifurcation theory in \cite{korman03} and recently by a variational approach in \cite{ItMU15}. In the fully-nonlinear case, that is $k>1$ was recently study in \cite{SaVe15} subject to conditions $n>2k$ and $q$ greater or equal to Tso's critical exponent and $\sigma=0$ (see \eqref{q:ast:kalfa} below). In \cite{korman03} for $q<\frac{n+\sigma}{n-2}$, was proved that there exists $\lambda_0>0$, so that the problem \eqref{Eq:f:pol} has exactly 2, 1 or 0 positive radial solutions, depending on whether $\lambda<\lambda_0$, $\lambda=\lambda_0$ or $\lambda>\lambda_0$. Notice that in case $\sigma>2$ this result cover in particular supercritical nonlinearities, i.e., the ones with $q>\frac{n+2}{n-2}$.

We recall that the $k$-Hessian operator in radial coordinates can be written as $S_k(D^2u)=c_{n,k}\,r^{1-n}\left(r^{n-k}(u')^k \right)'$, where $r=|x|,\,x\in\RR^n$ and where $c_{n,k}$ is defined by $c_{n,k}=\binom{n}{k}/n$ being $\binom{n}{k}$ the binomial coefficient.

Next, in order to state our main result, we write \eqref{Eq:f:pol} in radial coordinates, i.e.,
\begin{equation*}
(P_{\lambda})\qquad
\begin{cases}
c_{n,k}r^{1-n}\left(r^{n-k}(u')^k \right)' = \lambda\,r^\sigma (1-u)^q\,,\quad 0<r<1,\\
u(r)  < 0 \,, \hspace{4.65cm} 0<r<1,\\
u'(0)  =0,\, u(1)=0. &
\end{cases}
\end{equation*}
Now we introduce the space of functions $\Phi_0^k$ defined on $\Omega=(0,1)$ as in \eqref{phi-k-0}, for problem $(P_{\lambda})$:
\[
\Phi_0^k=\{u\in C^2((0,1))\cap C^1([0,1]): \left(r^{n-i}(u')^i \right)'\geq 0\;\;\mbox{in}\;\;(0,1) ,\, i=1,...,k,\,u'(0)=u(1)=0\}.
\]
We note that the functions in $\Phi_0^k$ are non positive on $[0,1]$. However, if $\left(r^{n-i}(u')^i \right)'> 0$ for all $i=1,\ldots, k$, then any function in $\Phi_0^k$ is negative and strictly increasing on $(0,1)$. This in turn implies that, if we are looking for solutions of ($P_\lambda$) in $\Phi_0^k$, then the parameter $\lambda$ must be positive.

\begin{definition}
Let $\lambda>0$. We say that a function $u \in C([0,1])$ is:
\begin{itemize}
\item[(i)] a {\it classical solution} of ($P_\lambda$) if $u\in \Phi_0^k$ and the equation in ($P_\lambda$) holds;
\item[(ii)] an {\it integral solution} of ($P_\lambda$) if $u$ is absolutely continuous on $(0,1]$, $u(1)=0$, $\int_0^1 r^{n-k}(u'(r))^{k+1} dr <\infty$ and the equality
\[
c_{n,k}r^{n-k}(u'(r))^k = \lambda\int_0^r s^{n-1+\sigma}(1-u(s))^qds,\, \, \text{ a.a. } r\in (0,1),
\]
holds whenever the integral exists.
\end{itemize}
\end{definition}

The concept of integral solution was introduced in \cite{ClFM96} for a more general class of radial operators, see e.g. \cite{ClFM96} and the references therein.  The standard concept of weak solution is equivalent in this case to the notion of integral solution, see \cite[Proposition 2.1]{ClFM96}.

\medbreak

The main goal of this paper is to describe the set of negative bounded radially symmetric solutions to \eqref{Eq:f:pol} in terms of the parameters. Our statements contain some classical and recent results (i.e. $k\geq 1$ and $\sigma= 0$), see \cite{JoLu73, SaVe15}, where the Endem-Fowler transformation was used to have a dynamical system, which allows to obtain the desire results. The key to stablish an Emden-Fowler transformation is finding an explicit singular solution $U$ of equation in \eqref{Eq:f:pol} on $\RR^n$, that is
\begin{equation}\label{singusolu}
U(x)=-|x|^{-\frac{2k+\sigma}{q-k}}\;\;\;\forall x\in \RR^n\setminus \{0\},
\end{equation}
corresponding to the parameter $\lambda =\tilde{\lambda}(k,\sigma)$ (see \eqref{Lambda:Tilde:kalfa}). After a straightforward (and lengthy) computation, we obtain the corresponding dynamical system associated to the Emdem-Fowler transformation, which looks more complicated comparing with a quadratic dynamical system. Our approach make use of a suitable change of variable, which transforms problem $(P_\lambda)$ into an equivalent two-dimensional Lotka-Volterra system. In order to illustrate this approach we use the model equation in the three-dimensional space
\begin{equation}\label{LaEmFow}
\frac{1}{r^2}\frac{d}{dr}\left(r^2\frac{d\phi}{dr}\right)=-r^{2m}\phi^{p+m},\; r>0,
\end{equation}
where $p\in (\frac{1}{2},\infty), m\in (-1,\infty)$ and $p+m>0$. Equation \eqref{LaEmFow} is known as a Generalized Lane-Emdem-Fowler equation \cite{BroVer82}. We note that when $m=0$ equation \eqref{LaEmFow} becomes the classical Lane-Emden equation of index $p$, which arises from stellar dynamic models (see  \cite{BroVer82}). Now we consider the following change of variables
\begin{equation}\label{Milne}
x=-\frac{r^{2m+1}\phi^{p+m}}{\phi'},\;\;y=-\frac{r\phi'}{\phi},
\end{equation}
\begin{equation}\label{BrVe}
r=e^t.
\end{equation}
The change of variables \eqref{Milne} with $m=0$ was introduced by Milne in the early thirties, see \cite{Milne30, Milne32} and also \cite{Chandrasekhar85}.

The authors Van den Broek and Verhulst studied equation \eqref{LaEmFow} using the change of variables \eqref{Milne} and introduced the key change of variable \eqref{BrVe}, which transforms the equation \eqref{LaEmFow} into the equivalent Lotka-Volterra system
\begin{equation*}
\begin{cases}
\frac{dx}{dt}= x[2m+3-x-(p+m)y],
\,\\
\frac{dy}{dt} = y[-1+x+y].
\end{cases}
\end{equation*}
This approach has been used also by other authors, see e.g. \cite{BaPf88, Wola99} and recently \cite{BatLi10, BidGia10, WaZhLi12}.

\medbreak

The paper is organized as follows: Section 2 is devoted to establish the new change of variables to reach a Lotka-Volterra system and to establish the main results, Theorem \ref{Main:1:Thm1} and Theorem \ref{Main:2:Intro}. The second theorem is proved in this section. In Section 3, we prove some general results on existence and nonexistence of solutions of problem $(P_\lambda)$. In Section 4, we identify the class of our Lotka-Volterra system. In Section 5 we make the local analysis of the phase portraits. Finally, in Section 6 we give the proof of Theorem \ref{Main:1:Thm1}. 

\section{New variables and main results}

\medbreak

We consider the radial version of problem \eqref{Eq:f}, as follows
\begin{equation}\label{RaPr:1}
\begin{cases}
\left(r^{n-k}(u')^k\right)'= r^{n-1}f(r,u),& 0<r<1,\\
u(r)<0, &0<r<1,\\
u'(0) = 0,\, u(1) =0.
\end{cases}
\end{equation}
Suppose that the function $f(r,u)$ is of the form
\[
f(r,u)=c_{n,k}^{-1}\,h(r)(1-u)^q,
\]
where $h\in C^1(0,1)$ and $h\neq 0$ on $(0,1)$. Let $u$ be a solution of \eqref{RaPr:1}. We define the function $w=u-1$. Then we see that $w$ solves the equation
\begin{equation}\label{Eq:IVP:0}
\left(r^{n-k}(w')^k\right)'= r^{n-1}h(r)(-w)^q, \;\, r>0.
\end{equation}
Now we are in position to state new variables to obtain a Lotka-Volterra system, which are given by
\begin{equation}\label{newtrans0}
x(t)=r^k\frac{h(r)(-w)^q}{(w')^{k}},\; y(t)=r\frac{w'}{-w}, \;r=e^t,
\end{equation}
where $w'$ stands for $dw/dr$. We point out that this change of variable is well-known in case $k=1$, see e.g. \cite{BaPf88, Wola99, BatLi10, WaZhLi12}. However, in the framework of the $k$-Hessian operator, this transformation seems to be unknown in the literature. Further, we see that such change of variable becomes optimal for problem \eqref{Eq:IVP:0}, since depending on the weight $h$ we obtain either an autonomous or non autonomous Lotka-Volterra system. More precisely, after some calculation one can see that the couple of functions $(x(t), y(t))$ solves the following (non autonomous) Lotka-Volterra system:
\begin{equation}\label{LVS0}
\begin{cases}
\frac{dx}{dt}=x\left[\rho(t)-x-q y\right],
\,\\
\frac{dy}{dt} = y\left[-\frac{n-2k}{k}+\frac{x}{k}+y\right],
\end{cases}
\end{equation}
where $\rho(t)=n+r\frac{h'(r)}{h(r)}$.

Now, in order to transform the problem $(P_\lambda)$ into a Lotka-Volterra system \eqref{LVS0}, we set $h(r)=c_{n,k}^{-1}\lambda r^\sigma$, obtaining the autonomous dynamical system:
\begin{equation}\label{LVS1}
\begin{cases}
\frac{dx}{dt}=x\left[n+\sigma-x-q y\right],
\,\\
\frac{dy}{dt} = y\left[-\frac{n-2k}{k}+\frac{x}{k}+y\right].
\end{cases}
\end{equation}

Note that we can recover $w$ via the formula
\begin{equation}\label{inverse0}
w=-\left[r^{2k}h(r)\right]^{-\frac{1}{q-k}}(xy^k)^{\frac{1}{q-k}}.
\end{equation}

\medbreak

We note that the existence and multiplicity results for $(P_{\lambda})$ are strictly related to the behavior of solutions of this dynamical system by \eqref{inverse0}. The advantage to transform problem $(P_{\lambda})$ into a quadratic dynamical system \eqref{LVS1} lies in the fact that such systems has been extensively studied, see e.g. \cite{Copp66, Baut54, ChTi82, Reyn07}. Even when there is no a complete classification for general quadratic systems, we found a classification of phase portraits according to the space of coefficients for the particular class of Lotka-Volterra system \eqref{LVS1}, see section 4 below.

On the other hand, by the definition of the new variables \eqref{newtrans0}, the region of interest is $\RR^2_+:=\{(x,y)\in\RR^2:x\geq 0,y\geq 0\}$ (a radial solution of $(P_{\lambda})$ is negative and its radial derivative is positive). In this region we find four critical points:  $(0,0), (0,\frac{n-2k}{k}), (n+\sigma,0)$, and
\begin{equation}\label{interiorcritipoint}
(\hat{x},\hat{y}):=\left(\frac{q(n-2k)-(n+\sigma)k}{q-k}, \frac{2k+\sigma}{q-k}\right).
\end{equation}
Note that according to our general assumptions, i.e. $2k<n$ and $k<q$, the first three critical points belong to $\RR^2_+$. The fourth critical point $(\hat{x},\hat{y})$ belongs to the interior of $\RR^2_+$ if, and only if, $q>(n+\sigma)k/(n-2k)$. It is not difficult to see that the orbit $(x(t),y(t))$ of \eqref{LVS1} starts from $(n+\sigma,0)$, see section 5.

\medbreak

Next, in order to estate our main result we introduce a critical exponent of the Joseph-Lundgren type, defined by
\begin{equation}\label{Exp:critical:Intro}
q_{JL}(k,\sigma)=
\begin{cases}
k\frac{k(k+1)n-k^2(2-\sigma)+2k+\sigma-2\sqrt{k(2k+\sigma)[(k+1)n-k(2-\sigma)]}}{k(k+1)n-2k^2(k+3)-2k\sigma-2\sqrt{k(2k+\sigma)[(k+1)n-k(2-\sigma)]}}, & n>2k+8+\frac{4\sigma}{k},\\
\infty, & 2k < n \leq 2k+8+\frac{4\sigma}{k}.
\end{cases}
\end{equation}
The Joseph-Lundgren exponent, i.e.,
\[
q_{JL}(1,0)=\frac{n-2\sqrt{n-1}}{n-4-2\sqrt{n-1}}
\]
was introduced in \cite{JoLu73}. We mention that the exponent $q_{JL}(1,\sigma)$ coincides with the critical exponent founded in \cite{DaDG11} in the study of the solutions of the problem $-\Delta u=|x|^\sigma |u|^{p-1}u$\, in \,$\Omega\subset\RR^n (n\geq 2)$ where $p>1, \sigma>-2$, and $\Omega$ is a suitable domain. We prove that $q_{JL}(k,\sigma)$ plays the same role as the Joseph-Lundgren exponent, that is, as soon as the critical exponent $q_{JL}(k,\sigma)$ is crossed, a drastic change in the number of bounded solutions of \eqref{Eq:f:pol} occurs, see Theorem \ref{Main:1:Thm1} below.

Another important exponent appearing in our main result is given by
\begin{equation}\label{q:ast:kalfa}
q^*(k,\sigma)=\frac{(n+2)k+\sigma(k+1)}{n-2k},
\end{equation}
which is smaller than $q_{JL}(k,\sigma)$. Further $q^*(k,\sigma)$ is bigger than $(n+\sigma)k/(n-2k)$ which ensures that $(\hat{x},\hat{y})$ belongs to the interior of $\RR^2_+$ for all $q\geq q^*(k,\sigma)$. The value $q^*(k,0)$ is well-known as the critical exponent in the study of the quasilinear $k$-Hessian operator, see \cite{Tso90} for more details. The exponent $q^*(1,\sigma)$ is called the Hardy-Sobolev exponent in the study of the Hardy-H\'{e}non equation $-\Delta u=|x|^{\sigma}u^p$, see \cite{PhSo12}.

Now we state our main result.
\begin{theorem}\label{Main:1:Thm1} Let $q>k$ and $n>2k$. Let $q^*(k,\sigma)$ and $q_{JL}(k,\sigma)$ be as in \eqref{Exp:critical:Intro} and \eqref{q:ast:kalfa}, respectively. Then there exists $\lambda^\ast>0$ such that problem $(P_{\lambda})$ admits a maximal bounded solution for $\lambda \in (0, \lambda^\ast)$, admits at least one possible unbounded integral solution for $\lambda = \lambda^\ast$, and it is no classical solutions for every $\lambda > \lambda^\ast$.

Moreover,
\begin{itemize}
	\item[(I)] If $q^\ast(k,\sigma) < q < q_{JL}(k,\sigma)$ and $\lambda$ is close to but not equal to
	\begin{equation}\label{Lambda:Tilde:kalfa}
\tilde{\lambda}(k,\sigma):= c_{n,k}\,\tau_{\sigma}^k(n-2k-k\tau_\sigma),	
	\end{equation}
	where $\tau_\sigma:=\frac{2k+\sigma}{q-k}$, then $(P_\lambda)$ has a large (finite) number of solutions. In addition, if $\lambda=\tilde{\lambda}(k,\sigma)$ then there exists infinitely many solutions of $(P_\lambda)$.
	\item[(II)] If $n>2k+8+\frac{4\sigma}{k}$, $q\geq q_{JL}(k,\sigma)$ and $\lambda\in (0, \lambda^*)$, then there exists only one solution of $(P_\lambda)$. Moreover, $\lambda^*=\tilde{\lambda}(k,\sigma)$.
\end{itemize}
\end{theorem}


Similar results are well-known in the literature in case $k=1$ and $\sigma=0$ see e.g. \cite{JoLu73, DoFl07}. In case $q>k$ and $\sigma=0$ problem $(P_{\lambda})$ was recently studied in \cite{SaVe15} by a suitable Emden-Fowler transformation. 

Observe that the function $u(x)= 1+U(x)$ with $x\in B\setminus\{0\}$ and $U$ as in \eqref{singusolu} is an integral solution of problem \eqref{Eq:f:pol} corresponding to the parameter $\lambda = \tilde{\lambda}(k,\sigma)$ provided that $q>q^*(k,\sigma)$ holds. Note that this integral solution can be easily obtained from the constant solution $(\hat{x}, \hat{y})$ of \eqref{LVS1} via the formula \eqref{inverse0}.

Note that the effect of multiplying the nonlinearity by a weight increases the critical exponents $q^\ast(k,0)$ and $q_{JL}(k,0)$ introduced in \cite{SaVe15}. In particular, $q^\ast(k,0)$ is shifted to the value $q^\ast(k,\sigma)$.

\medbreak

Now we first note that system \eqref{LVS1} has nontrivial closed orbits in the first quadrant if, and only if, $k<q$ and
\begin{equation}\label{center}
-\frac{n-2k}{k}(q+1)+(n+\sigma)\left(\frac{1}{k}+1\right)=0,
\end{equation}
by \cite[Theorem 2.1]{CaJi08}. Moreover, $(\hat{x},\hat{y})$ is a center. Further, if $\frac{(n+\sigma)k}{n-2k}<q$, then the center \eqref{interiorcritipoint} belongs to the interior of the first quadrant. Observe that the unique solution $q$ of \eqref{center} is exactly $q^\ast(k,\sigma)$, which is the critical exponent defined in \eqref{q:ast:kalfa}.

We consider the critical exponent problem
\begin{equation}\label{Eq:critical:alfa}
\begin{cases}
S_k(D^2u)= \lambda\,|x|^\sigma (1- u)^{q^*(k,\sigma)}&\mbox{in }\;\; B,\\
u <0 & \mbox{in }\;\; B,\\
u=0 &\mbox{on }\; \partial B.
\end{cases}
\end{equation}
The corresponding Lotka-Volterra system is given by
\begin{equation}\label{LVS:critical}
\begin{cases}
\frac{dx}{dt}=x\left[n+\sigma-x-q^*(k,\sigma) y\right],
\,\\
\frac{dy}{dt} = y\left[-\frac{n-2k}{k}+\frac{x}{k}+y\right].
\end{cases}
\end{equation}
It is remarkable that the line
\begin{equation}\label{invaline:0}
\frac{n-2k}{k}\,x+(n+\sigma)y-\frac{n-2k}{k}\,(n+\sigma)=0,
\end{equation}
which connects the critical points $(n+\sigma,0)$ and $(0,\frac{n-2k}{k})$, is an orbit of system \eqref{LVS:critical}. Moreover, from \eqref{invaline:0} we can obtain an explicit solution $(x(t),y(t))$ of system \eqref{LVS:critical}:

\begin{equation}\label{xysolu:0}
x(t)=(n+\sigma)\,\frac{c}{c+e^{\frac{2k+\sigma}{2k}\,t}},\, y(t)=\frac{n-2k}{k}\frac{e^{\frac{2k+\sigma}{2k}\,t}}{c+e^{\frac{2k+\sigma}{2k}\,t}},
\end{equation}
where $c$ is a positive constant. Since that $(x(t),y(t))$ is a bounded solution of \eqref{LVS:critical}, we may use \eqref{inverse0} to obtain a bounded solution $w$ of \eqref{Eq:IVP:0} with $h(r)=c_{n,k}^{-1}\,\lambda\, r^\sigma$. Indeed, replacing \eqref{xysolu:0} in \eqref{inverse0} with $q=q^\ast(k,\sigma)$, we get
\begin{equation}\label{wsolu:0}
w_c=-\lambda^{-\frac{n-2k}{(2k+\sigma)(k+1)}}\frac{[c\binom{n}{k}
\frac{n+\sigma}{n}(\frac{n-2k}{k})^k]^{\frac{n-2k}{(2k+\sigma)(k+1)}}}{(c+r^\frac{2k+\sigma}{k})^{\frac{n-2k}{2k+\sigma}}}.
\end{equation}
Note that $w_c$ is a scalar factor of a Bliss function, see \cite{Blis30}. Further, note that $w_c(\cdot)$ solves the problem
\begin{equation}\label{bliss:solution}
S_k(D^2u)=\lambda |x|^\sigma(- u)^{q^\ast(k,\sigma)}\,\, \mbox{ in }\,\, \RR^n,
\end{equation}
for all $c>0$. Now, restricting $w_c$ to the unit ball and setting $u=1+w_c$ we have that $u$ is a solution of \eqref{Eq:critical:alfa} if, and only if, there exists a value $d>0$ such that
\begin{equation}\label{Eq:d:alfa:0}
\lambda (d+1)^{k+1} - \binom{n}{k}\left(\frac{n-2k}{k}\right)^k\frac{n+\sigma}{n}\,d^k = 0,
\end{equation}
with $d=c^{-1}$. We can now verify by elementary calculus that \eqref{Eq:d:alfa:0} has either a unique solution $d = k$ if
\begin{equation}\label{Eq:mu:ast}
\lambda =  \binom{n}{k}\frac{n+\sigma}{n}\frac{(n-2k)^k}{(k+1)^{k+1}}:=\mu^\ast(k,\sigma)
\end{equation}
exactly two solutions (depending on $\lambda$ and $\sigma$) $d_- <d_+$ if $0<\lambda < \mu^\ast(k,\sigma)$ and no solutions $d$ if $\lambda > \mu^\ast(k,\sigma)$.

Hence problem \eqref{Eq:critical:alfa} has a solution if, and only if, $\lambda\leq \mu^*(k,\sigma)$. If $\lambda< \mu^*(k,\sigma)$, there exist exactly two solutions of \eqref{Eq:critical:alfa} given by
\begin{equation}\label{Eq:d-+}
v_{\lambda,\sigma} = w_{d_-}|_{B},\;\; V_{\lambda,\sigma} = w_{d_{+}}|_{B}.
\end{equation}
On the other hand, if $\lambda = \mu^*(k,\sigma)$, we have a unique solution of \eqref{Eq:critical:alfa} given by
\begin{equation}\label{Eq:dk}
V_{\sigma}^*=w_k|_{B}.
\end{equation}

Therefore, we conclude that problem \eqref{Eq:critical:alfa} has exactly two solutions $u_{\lambda,\sigma}$, $U_{\lambda,\sigma}$ if $\lambda< \mu^*(k,\sigma)$ and a unique solution $u_{\sigma}^*(k,\sigma)$ if $\lambda=\mu^*(k,\sigma)$, where
\[
u_{\lambda,\sigma}(x)=1-\lambda ^{-\frac{n-2k}{(2k+\sigma)(k+1)}}(-v_{\lambda,\sigma}(x)) ,\;\, U_{\lambda,\sigma}(x)=	1-\lambda ^{-\frac{n-2k}{(2k+\sigma)(k+1)}}(-V_{\lambda,\sigma}(x)).
\]
and
\[
u_{\sigma}^\ast(x) = 1-(\mu^\ast(k,\sigma))^{-\frac{n-2k}{(2k+\sigma)(k+1)}} (-V_{\sigma}^\ast(x)) = 1- \left(\frac{1+k}{1+k|x|^\frac{2k+\sigma}{k}}\right)^{\frac{n-2k}{2k+\sigma}}.
\]		

Thus we have proved the following result.
\begin{theorem}\label{Main:2:Intro}
Let $n>2k$. Consider the functions $v_{\lambda,\sigma}$, $V_{\lambda,\sigma}$, and $V_{\sigma}^*$  as in \eqref{Eq:d-+} and \eqref{Eq:dk}, respectively.  Let $\mu^\ast(k,\sigma)$ as in \eqref{Eq:mu:ast}.
\begin{itemize}
	\item[(i)] If $\lambda\in (0,\mu^\ast(k,\sigma))$, then there exist exactly two solutions of \eqref{Eq:critical:alfa} given by
	\[
u_{\lambda,\sigma}(x)=1-\lambda ^{-\frac{n-2k}{(2k+\sigma)(k+1)}}(-v_{\lambda,\sigma}(x)) ,\;\, U_{\lambda,\sigma}(x)=	1-\lambda ^{-\frac{n-2k}{(2k+\sigma)(k+1)}}(-V_{\lambda,\sigma}(x)).
	\]
	\item[(ii)] If $\lambda = \mu^\ast(k,\sigma)$, then \eqref{Eq:critical:alfa} has a unique solution given by
\[
u_{\sigma}^\ast(x) = 1-(\mu^\ast(k,\sigma))^{-\frac{n-2k}{(2k+\sigma)(k+1)}} (-V_{\sigma}^\ast(x)) = 1- \left(\frac{1+k}{1+k|x|^\frac{2k+\sigma}{k}}\right)^{\frac{n-2k}{2k+\sigma}}.
\]		
\end{itemize}
\end{theorem}

Notice that, for $q=q^*(k,\sigma)$, we have the estimate $\lambda^* \geq \mu^*(k,\sigma)$ by Theorem \ref{Main:1:Thm1} (see above). We mention that, in the case $k=1$ and $\sigma = 0$, the value $\mu{^\ast(1,0)}=\frac{n(n-2)}{4}$ coincides with the extremal value obtained in the classical paper \cite{JoLu73}. See also \cite{GaMa02} and \cite{Isse09}. In case $\sigma=0$, the value $\mu{^\ast(k,0)}$ coincides with the value obtained in \cite{SaVe15}.

\section{Existence and nonexistence of solutions of problem $(P_{\lambda})$}
We first note that equation \eqref{Eq:IVP:0} has the following scaling invariance property: let $\vartheta$ be a solution of \eqref{Eq:IVP:0} satisfying the initial conditions $\vartheta(0) = -1,\, \vartheta'(0) =0$, then the function $\vartheta_a(r):= a^{\delta}\vartheta(a r)$ with $\delta=(2k+\sigma)/(q-k)$ it is also a solution of \eqref{Eq:IVP:0} for all $a>0$.

Let $u$ be a bounded solution of \eqref{RaPr:1} with $f(r,u)=c_{n,k}^{-1}\lambda r^{\sigma}(1-u)^q$ and let $u(0)=A\in (-\infty,0)$. Let $w=u-1$ be a solution of \eqref{Eq:IVP:0}. We make a normalization by $v=w/(1-A)$ and the rescaling
\[
s=\left(\frac{\lambda}{\tilde{\lambda}(k,\sigma)}\right)^{\frac{1}{2k+\sigma}}\left(1-A \right)^{\frac{q-k}{2k+\sigma}}r,\, r>0.
\]
We obtain that $v$ solves the initial value problem
\begin{equation}\label{Eq:v:alfa}
\begin{cases}
\left(s^{n-k}(v')^k\right)'= s^{n-1}\bar{\lambda}(k,\sigma)\,s^\sigma(-v)^q, & s>0,\\
v(0) = -1,\, v'(0) =0,
\end{cases}
\end{equation}
by the boundary conditions of  \eqref{RaPr:1} and $\bar{\lambda}(k,\sigma):=c_{n,k}^{-1}\tilde{\lambda}(k,\sigma)$. Note that by the scaling invariance property is it enough focused our study on \eqref{Eq:v:alfa}. Further, the change of variable \eqref{newtrans0} applied to \eqref{Eq:v:alfa} yields with the same Lotka-Volterra system \eqref{LVS1}. Moreover we can write $v$ in terms of the solution $(x,y)$ of system \eqref{LVS1} and the new parameters, that is
\begin{equation}\label{v:s:t}
v=-\left[s^{2k+\sigma}\bar{\lambda}(k,\sigma)\right]^{-\frac{1}{q-k}}(xy^k)^{\frac{1}{q-k}}.
\end{equation} 

We will need the following two technical lemmas.

\begin{lemma}\label{Uniq:sol}
Let $q\geq q^\ast(k,\sigma)$. Then there exists a unique global solution $v$ of \eqref{Eq:v:alfa} in the regularity class $C^2(0,\infty)\cap C^1[0,\infty)$.
\end{lemma}
\begin{proof}
Let $\tau = \bar{\lambda}^{\frac{1}{2k+\sigma}}s$ and  set $z(\tau) = v(s)$. Then we may rewrite \eqref{Eq:v:alfa} as
\begin{equation}\label{Eq:w:aux}
\begin{cases}
\left(\tau^{n-k}(z')^k\right)'= \tau^{n-1+\sigma}(-z)^q, & \tau>0,\\
z(0) = -1,\\
z'(0) =0.
\end{cases}
\end{equation}
Defining $B(r)=\int_0^r s^{\frac{(\alpha-\beta)(q+1)}{\beta+1}}(a(s)s^{\theta})'ds$, $r>0$, with $\alpha=n-k$, $\beta=k$, $\gamma=n-1$, $a(s)=s^\sigma$ and $\theta=[(\gamma+1)(\beta+1)-(\alpha-\beta)(q+1)]/(\beta +1)$, we obtain $B(r)\leq 0$ for $r>0$ if, and only if, $q\geq q^*(k,\sigma)$. Then the global existence of \eqref{Eq:w:aux} follows from \cite[Theorem 4.1]{ClMM98}. Now, for the uniqueness we define the map $T: \mathcal{B} \to  \mathcal{B}$, where $ \mathcal{B}:=\{ z\in C[0,t_0]\;:\; z(0)=-1 \text{ and } |z+1| \leq 1/2\}$, by
\[
T(z)(r):= -1 + \int_0^r \left( \frac{1}{t^{n-k}}\int_0^t s^{n-1+\sigma}(-z(s))^q ds\right)^{\frac{1}{k}}dt,\; r\in [0,t_0].
\]
Using the arguments given in the proof of \cite[Lemma 4.1]{SaVe15}, we see that $T$ admits a unique fixed point by the contraction mapping principle.
\end{proof}

\begin{lemma}\label{max:sol}
Let $n > 2k$, $q > k$ and $\lambda_0 >0$. Assume that there exists a classical solution of
\begin{equation}
\begin{cases}\label{Eq:2:0}
c_{n,k}r^{1-n}\left(r^{n-k}(w')^k \right)' = \lambda_0\,r^{\sigma}(1-w)^q\,,\quad 0<r<1,\\
w  < 0 \,, \hspace{4.65cm} 0<r<1,\\
w'(0)  =0,\, w(1)=0, &
\end{cases}
\end{equation}

Then, for any $\lambda\in (0,\lambda_0)$, problem $(P_\lambda)$ has a maximal bounded solution. Moreover, the maximal solutions form a decreasing sequence as $\lambda$ increases.
\end{lemma}
\begin{proof}
Fix $\lambda\in (0,\lambda_0)$ and define the functions
\[
g(t) = \left[\lambda_0(1+t)^q\right]^{1/k}\, \text{ and }\;\, \tilde{g}(t) = \left[\lambda(1+t)^q\right]^{1/k},\text{ for all }\; t\geq 0.
\]
Set $\Phi(s) = \tilde{h}^{-1}(h(s))$ ($s\leq 0$) with $h$ and $\tilde{h}$ given by
\[
h(s) = \int_s^0 \frac{1}{g(-t)}\,dt\; \text{ and }\;\, \tilde{h}(s) = \int_s^0 \frac{1}{\tilde{g}(-t)}\,dt, \; s\leq 0.
\]
Since $q>k$, we have that $\lim_{s\to -\infty}h(s)$ exists and hence $\Phi$ is bounded by \cite[Lemma 2.1 (i)-(ii)]{SaVe15}. Next, by \eqref{Eq:2:0} and the convexity of $\Phi$ by \cite[Lemma 2.1 (iii)]{SaVe15} we have
\begin{eqnarray*}
S_k(D^2 \Phi(w))&=&c_{n,k}kr^{1-k}(\Phi'(w)w')^{k-1}\left(\Phi''(w)(w')^2+\Phi'(w)w''+\frac{n-k}{k}\frac{\Phi'(w)w'}{r}\right)\\
&\geq & c_{n,k}kr^{1-k}(\Phi'(w))^k(w')^{k-1}\left(w''+\frac{n-k}{k}\frac{w'}{r}\right)\\
& = &(\Phi'(w))^k S_k(D^2 w) = \frac{(\tilde{g}(-\Phi(w)))^k}{(g(-w))^k}S_k(D^2 w) = \lambda r^{\sigma}(1-\Phi(w))^q.
\end{eqnarray*}
Therefore,  $\Phi(w)$ is a bounded subsolution of $(P_\lambda)$ and hence by the method of super and subsolutions we have, by \cite[Theorem 3.3]{Wang94}, a solution $u\in L^{\infty}((0,1))$ of $(P_\lambda)$ with $\Phi(w) \leq u \leq 0$. Now, to prove that ($P_{\lambda}$) admits a maximal solution, we consider $u_1$ the solution of
\begin{equation*}
\begin{cases}
S_k(D^2 u_1)= \lambda |x|^{\sigma}&\mbox{in }\;\; B,\\
u_1=0 &\mbox{on }\; \partial B.
\end{cases}
\end{equation*}
Since $u$ is in particular a subsolution of ($P_{\lambda}$), we have $u\leq u_1$ on $B$ by the comparison principle, see \cite{TrXu99}. Next, we define $u_i$ ($i=2,3,\ldots$) as the solution of
\begin{equation*}
\begin{cases}
S_k(D^2 u_i)= \lambda |x|^{\sigma}(1-u_{i-1})^q &\mbox{in }\;\; B,\\
u_i=0 &\mbox{on }\; \partial B.
\end{cases}
\end{equation*}
Using again the comparison principle we obtain a decreasing sequence of $u_i$ bounded from below by $u$ and by 0 from above. Hence, we can pass to the limit and we obtain a solution $u_{max}$ of ($P_{\lambda}$), which is maximal since the recursive sequence $\{u_i\}$ does not depend on the subsolution $u$. Now, let $\lambda_1<\lambda_2$ and $u_{\lambda_1}$, $u_{\lambda_2}$ be maximal solutions of $(P_{\lambda_i})$ ($i=1,2$), respectively. Note that $u_{\lambda_2}$ is a subsolution of $(P_{\lambda_1})$, whence $u_{\lambda_2} \leq u_{\lambda_1}$ by the maximality of $u_{\lambda_1}$.
\end{proof}

For $R>1$, let $B_R$ be a ball centered at zero with radius $R$ such that $\overline{B}\subset B_R$ and let $\eta$ be the solution of
\begin{equation*}
\begin{cases}
S_k(D^2\eta)= 1 &\mbox{in }\;\; B_R,\\
\eta=0 &\mbox{on }\; \partial B_R.
\end{cases}
\end{equation*}
Then there exists a negative constant $\beta$ such that $\eta<\beta<0$ on $\partial{B}$.
Set $M=max_{x\in\overline{B}}\,|\eta|$ and take $\lambda<(1+M)^{-q}$. Then
\[
S_k(D^2\eta)=1>\lambda(1+M)^q\geq\lambda(1-\eta)^q\geq\lambda |x|^{\sigma}(1-\eta)^q \;\;\mbox{in}\;\;B.
\]
By \cite[Theorem 3.3]{Wang94}, for any $\lambda\in (0,(1+M)^{-q})$ there exists a solution $u_{\lambda}$ of $(P_\lambda)$. Thus we may define
\begin{equation}\label{lambda-ast}
\lambda^\ast=\sup\{\lambda>0:\, \text{ there is a solution } u_{\lambda}\in C^2(B) \text{ of } \eqref{Eq:f:pol}\}.
\end{equation}
Hence $\lambda^\ast>0$.

To see that $\lambda^\ast$ is finite we consider the inequality
\begin{equation}\label{Eq:Nineq}
\Delta u\geq C(n,k)[S_k(D^2 u)]^\frac{1}{k},\; u\in \Phi^k(B).
\end{equation}
See e.g. \cite{Wang09}. Consider the eigenvalue problem
\begin{equation*}
(E_m)\;\;
\begin{cases}
-\Delta u= \lambda m(x)u &\mbox{in }\;\; B,\\
u=0 &\mbox{on }\; \partial B,
\end{cases}
\end{equation*}
where $m(x):=|x|^{\frac{\sigma}{k}}$. It is known that problem $(E_m)$ has a first eigenvalue, $\lambda_{1,m}>0$, associated with an eigenfunction $\phi_{1,m}>0$. Let $\lambda\in (0,\lambda^\ast)$ and let $u$ be a solution of problem $(P_\lambda)$. Then, using \eqref{Eq:Nineq}, we obtain
\[
\Delta u\geq C(n,k)\lambda^{\frac{1}{k}}|x|^{\frac{\sigma}{k}}(1-u)(1-u)^{\frac{q-k}{k}}\geq C(n,k)\lambda^{\frac{1}{k}}|x|^{\frac{\sigma}{k}}(1-u)\geq C(n,k)\lambda^{\frac{1}{k}}|x|^{\frac{\sigma}{k}}(-u),
\]
which in turn implies $\lambda<\left(\frac{\lambda_{1,m}}{C(n,k)}\right)^k$. Thus $\lambda^{\ast}$ is finite.

Now, let $\lambda\in (0,\lambda^{\ast})$. Then $u_\lambda$ is a maximal bounded solution of ($P_{\lambda}$) by Lemma \ref{max:sol} applied with $\lambda_0 \in (\lambda, \lambda^{\ast})$.

\medbreak

Now let $\lambda_i$ be an increasing sequence such that $\lambda_i\rightarrow \lambda^\ast$ as $i\rightarrow +\infty$ and let $u_{\lambda_i}$ be a maximal solution of $(P_{\lambda_i})$. By Lemma \ref{max:sol}, for all $r\in [0,1]$, we have $u_{\lambda_{i+1}}(r)\leq u_{\lambda_{i}}(r)\leq 0$. On the other hand, integrating the equation in $(P_{\lambda_i})$, we obtain
\[
u_{\lambda_i}(r)=\int_{r}^{1}\left[c_{n,k}^{-1}\,\tau^{k-n}\int_{0}^{\tau}s^{n-1+\sigma}\lambda_{i}(1-u_{\lambda_i}(s))^q\,ds\right]^{\frac{1}{k}}d\tau.
\]
Now, applying twice the monotone convergence theorem we conclude that
\[
u^\ast(r):=\lim_{i\rightarrow +\infty}u_{\lambda_i}(r), \;\;\; \text{ exists a.a. } r\in  (0,1)
\]
and
\[
u^\ast(r)=\int_{r}^{1}\left[c_{n,k}^{-1}\,\tau^{k-n}\int_{0}^{\tau}s^{n-1+\sigma}\lambda^\ast(1-u^\ast(s))^q\,ds\right]^{\frac{1}{k}}d\tau, \;\;\; \text{ a.a. } r\in  (0,1).
\]
The assertion concerning the non existence of solutions follows directly from the definition of $\lambda^\ast$. 


\section{Classification of our Lotka-Volterra system} 
We use the notation of \cite{Reyn07} to define
\begin{align*}
P(x,y) &= (n+\sigma)x - x^2 - qxy\\
Q(x,y) & = -\frac{n-2k}{k}\,y + \frac{1}{k}\,xy + y^2
\end{align*}
Now in order to compute the Poincar{\'e} index of a (finite) critical point $(x_0,y_0)$ we define
\[
\Lambda (x_0,y_0) = \partial_x P(x_0,y_0)\partial_yQ(x_0,y_0)- \partial_y P(x_0,y_0)\partial_xQ(x_0,y_0).
\]
For system \eqref{LVS1} we have that the critical points $(0,0)$, $(n+\sigma, 0)$ and $(0,(n-2k)/k)$ are saddle points with Poincar{\'e} index $-1$. The fourth critical point $(\hat{x},\hat{y})$ is an antisaddle with Poincar{\'e} index $1$. According to \cite[Section 2.3.1]{Reyn07} the notation for the saddle points is $e^{-1}$ and for antisaddle points $e^1$. Thus system \eqref{LVS1} admits the combination $e^{-1}e^{-1}e^{-1}e^{1}$ of the finite critical points.

Next, in order to see the behavior of system \eqref{LVS1} near infinity we use polar coordinates to compute the critical points at infinity. System \eqref{LVS1} can be written as follows
\begin{align}
r' & =A_1(\theta) + rB_1(\theta)\\
\theta' & = A_2(\theta) + r B_2(\theta).
\end{align}
Since the critical points at infinity are characterized by the condition $\theta'=0$ (see \cite[Chapter 1]{Reyn07}). The dominant term at large $r$ is $B_2(\theta)$ which is given by $B_2(\theta) = \cos\theta \sin\theta[(1/k +1)\cos\theta + (q+1)\sin\theta]$. The solutions of $B_2(\theta)=0$ are given by $\theta=0(\pi)$, $\theta=\pi/2(3\pi/2)$, and $\theta=\arctan(-(k+1)(k(q+1)))$. Actually, this solutions correspond (modulo $\pi$) to three infinite critical points. Now, we use the Poincar{\'e} sphere. For this, we introduce the change of variable $z=1/x$ and $u=y/x (=\tan\theta)$. Thus system \eqref{LVS1} becomes
\begin{align}
zz'&=-z[-1+(n+\sigma)z -qu]=:\hat{P}(z,u)\label{sys:zz}\\
zu' & =\frac{k+1}{k}\,u - \left(n+\sigma + \frac{n-2k}{k}\right)zu + (q+1)u^2 =:\hat{Q}(z,u).\label{sys:zu}
\end{align}
The critical points at infinity are given by $(0,u)$, where $u$ is a solution of
\begin{equation}\label{critical:infty}
\frac{k+1}{k}\,u + (q+1)u^2=0.
\end{equation}
The eigenvalues of the linearized system \eqref{sys:zz}-\eqref{sys:zu} near infinity are given by
\begin{align*}
\lambda_z&=\partial_z\hat{P}(0,u)=1+qu\\
\lambda_u&=\partial_u\hat{Q}(0,u)=\frac{k+1}{k} + 2qu,
\end{align*}
where $u=0$, $u=-(k+1)/(k(q+1))$ are the solutions of \eqref{critical:infty}. Defining $\Lambda_c$ as the product of the eigenvalues of the linearized system \eqref{sys:zz}-\eqref{sys:zu}, it easy to see that $\Lambda_c > 0$ for all solution $u$ of \eqref{critical:infty}. For the third infinite critical point located on the $y$-axis at infinity, we rotate the axes such that the critical point under consideration is at the end of the $x$ axis so that $\Lambda_c = -1(-1-1/k)>0$, cf.\cite[Page 27]{Reyn07}. Hence system \eqref{LVS1} admits the combination $E^1E^1E^1$ of infinite critical points, where $E^1$ denotes a node. See \cite[Section 2.3.2]{Reyn07} for a more general presentation. Hence our system is classified in the class $e^{-1}e^{-1}e^{-1}e^{1}E^1E^1E^1$. This combination is studied in \cite[Sections  3.4.1 and 11.2.2]{Reyn07}. See \cite[Section 11.2.2, figure 11.5 $(\underline{a})$-$(\underline{b})$]{Reyn07} for the corresponding phase portraits.

Since the center case $q=q^\ast(k,\sigma)$ was already studied in Section 2, we focus only on orbits connecting the critical points $(n+\sigma,0)$ and the point $(\hat{x},\hat{y})$ located at the interior of the first quadrant.

\section{Local analysis at the point $(\hat{x},\hat{y})$}
This section contains the hardest part of this work. The main difficulties here are to obtain the critical exponent $q_{JL}(k,\sigma)$ and to define an auxiliary variable $a_{\sigma}$ defined below, which is the key to determinate the stability of our system. 

Next we show that the orbits of system \eqref{LVS1} start from $(n+\sigma,0)$ and end at $(\hat{x},\hat{y})$. By \eqref{newtrans0} and \eqref{Eq:v:alfa}, the function $y = y(t)$ satisfies
\begin{equation}\label{limi:y}
\lim_{t\to -\infty}y(t)=\lim_{r\to 0}r\,\frac{v'(r)}{-v(r)} = 0
\end{equation}
and for $x=x(t)$, we have
$
\lim_{t\to -\infty}x(t)=\lim_{r\to 0}\bar{\lambda}(k,\sigma)[-v(r)]^q\frac{r^{k+\sigma}}{[v'(r)]^k}.
$
Now, by $\eqref{Eq:IVP:0}$ and L'hospital's rule, we have
\[
\lim_{r\to 0}\frac{[v'(r)]^k}{ r^{k+\sigma}}=\lim_{r\to 0}\frac{\bar{\lambda}(k,\sigma)\int_{0}^{r}\tau^{n-1+\sigma}[-v(\tau)]^q d\tau}{r^{n+\sigma}}=\lim_{r\to 0}\frac{\bar{\lambda}(k,\sigma)r^{n-1+\sigma}[-v(r)]^q}{(n+\sigma)r^{n-1+\sigma}}=\frac{\bar{\lambda}(k,\sigma)}{n+\sigma}.
\]
Then
\begin{equation}\label{limi:x}
\lim_{t\to -\infty}x(t)=\frac{\bar{\lambda}(k,\sigma)(n+\sigma)}{\bar{\lambda}(k,\sigma)}=n+\sigma.
\end{equation}
From \eqref{limi:y} and \eqref{limi:x} we conclude that
\begin{equation}\label{start}
\lim_{t\to -\infty}(x(t),y(t))=(n+\sigma,0).
\end{equation}

\begin{remark}
The variables $x$ and $y$ in \eqref{newtrans0} are nonnegative since the unique global solution $v$ of \eqref{Eq:v:alfa} is negative and increasing by Lemma \ref{Uniq:sol}, which in turn imply that our study of the phase portrait must to be restricted only to positive orbits starting from $(n+\sigma,0)$ and ending in $(\hat{x},\hat{y})$ by the previous section.
\end{remark}

Since, we are looking for orbits that start from $(n+\sigma,0)$ and end in $(\hat{x},\hat{y})$, we only need a local analysis at the critical point $(\hat{x},\hat{y})$. To this end, we consider the linearization of \eqref{LVS1} at the critical point $(\hat{x},\hat{y})$. The Jacobian matrix at the point $(\hat{x},\hat{y})$ is given by
\[
J=\left(\begin{array}{cc}
-\frac{q(n-2k)-(n+\sigma)k}{q-k}& -\frac{q[q(n-2k)-(n+\sigma)k]}{q-k}\\
\,\\
\frac{2k+\sigma}{k(q-k)} & \frac{2k+\sigma}{q-k}
\end{array}
\right)
\]
The eigenvalues of $J$ are
$
\lambda_{\pm} =\frac{1}{2} \text{tr} J \pm \frac{1}{2} \sqrt{(\text{tr} J)^2 - 4 \text{det} J},
$
where $\text{tr} J = \frac{2k+\sigma-[q(n-2k)-(n+\sigma)k]}{q-k}$ and $\text{det} J = \frac{(2k+\sigma)[q(n-2k)-(n+\sigma)k]}{k(q-k)}$. The discriminant $\Delta$ is given by
\begin{equation}\label{DeltamodelI}
\varDelta(a_\sigma)=(q-k)^{-2}\left\{[(2k+\sigma)-a_\sigma]^2-4\frac{(2k+\sigma)(q-k)}{k}\,a_\sigma\right\},
\end{equation}
where
\begin{equation}\label{Eq:aalfa}
a_\sigma:=q(n-2k)-(n+\sigma)k.
\end{equation}

The location of the eigenvalues $\lambda_{\pm}$ on the complex plane is determined as follows:
\begin{itemize}
	\item[(i)] If $(2k+\sigma)-a_\sigma>0$ and $\varDelta \geq 0$, then the eigenvalues $\lambda_{\pm}$ are real positive numbers.
	\item[(ii)] If $(2k+\sigma)-a_\sigma>0$ and $\varDelta <0$, then the eigenvalues $\lambda_{\pm}$ are complex numbers with positive real part.
	\item[(iii)] If $(2k+\sigma)-a_\sigma<0$ and $\varDelta <0$, then the eigenvalues $\lambda_{\pm}$ are complex numbers with negative real part.
	\item[(iv)] If $(2k+\sigma)-a_\sigma<0$ and $\varDelta \geq0$, then the eigenvalues $\lambda_{\pm}$ are negative real numbers.
	\item[(v)] If $(2k+\sigma)-a_\sigma=0$, then the eigenvalues $\lambda_{\pm}$ are purely imaginary.
\end{itemize}

The case $(2k+\sigma)-a_\sigma>0$ corresponds to instability, $(2k+\sigma)-a_\sigma<0$ to stability, and for $(2k+\sigma)-a_\sigma=0$ to a center. We point out that $(2k+\sigma)-a_\sigma=0$ is equivalent to defining $q=q^\ast(k,\sigma)$. We first note that $(2k+\sigma)-a_\sigma<0$ is equivalent to $q>q^*(k,\sigma)$ and by Lemma \ref{Uniq:sol} there exists a unique classical global solution of \eqref{Eq:v:alfa}. In case $(2k+\sigma)-a_\sigma>0$ (i.e. $q<q^*(k,\sigma)$) we cannot claim the existence of a global solution of  \eqref{Eq:v:alfa} and thus the discussion about this case is excluded. In the stability case (iii)-(iv), that is $q>q^*(k,\sigma)$, depending on the sign of the discriminant $\varDelta$ we obtain two types of orbits a spiral or a stable node. Using the same arguments given in \cite{SaVe15} ($\sigma =0$) we can stablish that for $q^*(k,\sigma) < q < q_{JL}(k,\sigma)$ we have a spiral and for $q_{JL}(k,\sigma)\leq q$ a stable node. In fact, solving the equation
\begin{equation}\label{f:Intro}
n-2k=f_{k,\sigma}(q),
\end{equation}
where
\[
f_{k,\sigma}(q)=\frac{2q(2k+\sigma)}{k(q-k)}+\frac{2(2k+\sigma)}{k}\sqrt{\frac{q}{q-k}}+\frac{(2k+\sigma)(k-1)}{q-k},
\]
in $q$, we obtain the unique solution $q_{JL}(k,\sigma)$ of \eqref{f:Intro} and replacing it in \eqref{Eq:aalfa} and \eqref{DeltamodelI}, gives $\Delta=0$. Observe that, for $k=1$ and $\sigma=0$, $f_{1,0}(q)$ coincides with the function $f$ introduced in \cite{JoLu73}. See \cite{SaVe15} for more details in case $\sigma=0$.

\begin{lemma}\label{lem:qJL:nodo}
Let $n>2k+8 + 4\sigma/k$ and $q \geq q_{JL}(k,\sigma)$. Then the unique solution $(x(t),y(t))$ of \eqref{LVS1} coincides with the graph of an increasing function $y = y(x)$
\end{lemma}
\begin{proof}
The orbit starts at the critical point $(n+\sigma, 0)$ with a slope given by
\begin{equation}\label{Eq:gamma_s}
\gamma_s:=-\frac{(n-2k)q^*(k,\sigma)}{qk(n+\sigma)},
\end{equation}
by L'Hospital's rule.

Now we describe the behavior of the orbit $(x(t),y(t))$ near $(\hat{x},\hat{y})$ as time goes to infty. To this end, we compute
\[
\gamma:=\lim_{t\to +\infty}\frac{y'(t)}{x'(t)}
\]
again by L'Hospital's rule. The existence of the limit above is equivalent to solving the equation
\begin{equation}\label{gamma:def}
\gamma^2 +\frac{2k+\sigma + a_{\sigma}}{qa_{\sigma}} \gamma +\frac{2k+\sigma}{qka_{\sigma}}=0,
\end{equation}
whose roots are given by
\begin{equation*}
\gamma_{\pm} = - \frac{2k+\sigma + a_{\sigma}}{2qa_{\sigma}} \pm \frac{1}{2qa_{\sigma}}\sqrt{(2k+\sigma + a_\sigma)^2-4\frac{(2k+\sigma)q}{k}\,a_{\sigma}}.
\end{equation*}
Note that $(q-k)^2\Delta(a_{\sigma}) = (2k+\sigma + a_\sigma)^2-4\frac{(2k+\sigma)q}{k}\,a_{\sigma}$, where $\Delta(a_{\sigma})$ is as in \eqref{DeltamodelI}. Hence the roots of \eqref{gamma:def} are real numbers if, and only if, $q\geq q_{JL}(k,\sigma)$. Further, $\gamma_{\pm}$ are negative roots.

Next, consider a function $g$ defined by
\begin{equation}\label{Def:g}
g(x) = d(n+\sigma- x)^{\alpha}, \, x\in [\hat{x}, n+\sigma]
\end{equation}
where the constants $d$ and $\alpha$ are chosen such that the graph of $g$ connects the points $(n+\sigma, 0)$ and $(\hat{x},\hat{y})$. Using $\hat{y}=g(\hat{x})$, we obtain
\[
d=q^{-\alpha}\left(\frac{2k+\sigma}{q-k}\right)^{1-\alpha}.
\]
Now, setting $g'(\hat{x})=(\gamma_{+}+\gamma_{-})/2$, we have $\alpha=\frac{2k+\sigma + a_{\sigma}}{2 a_{\sigma}}$. Note that $\alpha \in (1/2, 1)$ since $q\geq q_{JL}(k,\sigma)$ and $2k+\sigma - a_{\sigma} = (n-2k)(q^*(k,\sigma) - q)<0$.

We claim that $(x(t), y(t))$ remains below the curve $y=g(x)$ when $x\in (\hat{x},n+\sigma)$. Indeed, we first note that the curve $(x(t), y(t))$ lies above the line $y=-x/q + (n+\sigma)/q$ since the slope $-1/q$ is bigger than $\gamma_s$ in \eqref{Eq:gamma_s} and remains from below $(x,g(x))$ near $(n+\sigma, 0)$ on the phase plane.

Suppose now by contradiction that $(x(t), y(t))$ intersects the curve $y=g(x)$ in a point $(x_0,y_0)$ at $t=t_0$. In this case we have two possibilities: the orbit  $(x(t), y(t))$ remains on the graph of $g$ for all $t\geq t_0$ and then $(x(t), y(t))$ arrives at $(\hat{x},\hat{y})$ with the same slope that $y=g(x)$ at $\hat{x}$, which is impossible since  $g'(\hat{x})=(\gamma_{+}+\gamma_{-})/2 \neq \gamma_{\pm}$ because the inequality $q>q_{JL}(k,\sigma)$. The other case is that the orbit crosses the graph of $g$ at point $x_0$. In this case we have
\begin{equation}\label{g:prima}
g'(x_0)> \frac{y'(t_0)}{x'(t_0)}.
\end{equation}
We now show that \eqref{g:prima} is impossible. To this end, define the functions
\begin{equation}\label{F_1:F_2}
\begin{cases}
F_1(x,y)= x(n+\sigma - x - qy),\\
\,\\
F_2(x,y) = y(-\frac{n-2k}{k} +\frac{1}{k}x + y),\\
\end{cases}
\end{equation}
for all $(x,y)\in [\hat{x},n+\sigma]\times [0,\hat{y}]$. Note that
\begin{equation}\label{F:at:y0}
F_2(x_0,g(x_0)) - g'(x_0)F_1(x_0,g(x_0)) = F_2(x_0,y_0) - g'(x_0)F_1(x_0,y_0) <0.
\end{equation}
Next, set
\begin{equation}\label{h:def:F}
h(x) = F_2(x,g(x)) - g'(x)F_1(x,g(x)), \, x\in (\hat{x},n+\sigma).
\end{equation}
Then, by \eqref{Def:g} and \eqref{F_1:F_2} we have
\[
h(x) = c_1(n+\sigma-x)^{\alpha} + c_2\,x(n+\sigma - x)^{\alpha} + c_3(n+\sigma -x)^{2\alpha}+c_4\,x(n+\sigma -x)^{2\alpha -1},
\]
where\; $c_1:=-\frac{n-2k}{k}\,d,\, c_2:=\frac{k\alpha+1}{k}\,d,\, c_3:=d^2$\; and\, $c_4:= -\alpha q d^2$.
\medbreak

Using \eqref{F_1:F_2}, we obtain
\begin{equation}{\label{h:prima:y2}}
h'(\hat{x}) = \frac{qa_{\sigma}}{q-k}\left(\beta^2 + \frac{2k+\sigma + a_\sigma}{qa_{\sigma}}\beta + \frac{2k + \sigma}{kqa_{\sigma}}\right),
\end{equation}
where $\beta=g'(\hat{x})$. Further, since $\beta=(\gamma_+ + \gamma_-)/2$ we have $h'(\hat{x})\leq 0$ by \eqref{gamma:def}, with the strict inequality if $q > q_{JL}(k,\sigma)$ and the equality if $q = q_{JL}(k,\sigma)$ holds. Now note that $h(n+\sigma)=0$ since $\alpha\in(1/2,1)$ and $h(\hat{x})=0$ by \eqref{h:def:F} and the equality $\hat{y}=g(\hat{x})$. The derivative of $h$ may be written as
\begin{equation}\label{hprima}
h'(x)= d(n+\sigma - x)^{\alpha-1}\rho(x)
\end{equation}
being the function $\rho$ given by
\[
\rho(x)=  c_5(n+\sigma -x) - \left(2 + \alpha x + \frac{\sigma}{k}\right)\alpha +c_6(n+\sigma -x)^{\alpha} +c_7\,x(n+\sigma -x)^{\alpha -1},
\]
where\; $c_5:=\alpha+(\alpha+1)/k,\, c_6:=-\alpha (q+2)d$\; and\; $c_7:=\alpha(2\alpha-1)qd$.
\medbreak

Now, in order to determine the growth of $h$ near $n+\sigma$, we rewrite $h'$ as
\[
h'(x)=d(n+\sigma - x)^{1-\alpha}\rho(x)/(n+\sigma - x)^{2(1-\alpha)}
\]
and conclude that
\begin{equation}\label{h:prima:0}
\lim_{x\,\uparrow\, n+\sigma} h'(x) = \infty,
\end{equation}
since $\alpha\in(1/2,1)$. Note that $\rho(\hat{x})\leq 0$ by \eqref{h:prima:y2}-\eqref{hprima}. In particular, $\rho(\hat{x})=0$ in the case $q=q_{JL}(k,\sigma)$ by \eqref{h:def:F} and the equalities $g'(\hat{x})=\gamma_+=\gamma_- = \gamma$. On the other hand,  it is easy to see that
\begin{equation}\label{rho:0}
\lim_{x\,\uparrow\, n+\sigma}  \rho(x) = \infty.
\end{equation}
Computing $\rho'(x)$ we conclude that the equation $\rho'(x)=0$ admits only one solution which is not a minimum by \eqref{rho:0}. Further, the equation $\rho(x)=0$ has at most two solutions on $[\hat{x},n+\sigma)$ since, if there is no solution, then $\rho>0$ on $[\hat{x},n+\sigma)$ by \eqref{rho:0} and hence by \eqref{hprima} we deduce that $h$ is increasing on $[\hat{x},n+\sigma)$ contradicting  $h(\hat{x})=h(n+\sigma)=0$. Now, if we have two solutions on $(\hat{x},n+\sigma)$ then $\rho(\hat{x})>0$, which is a contradiction since $\rho(\hat{x})\leq 0$. Hence the equation $\rho(x)=0$ admits at most two solutions on $[\hat{x},n+\sigma)$. In this case we see that $h$ admits at most two critical points on $[\hat{x},n+\sigma)$. But this is impossible since $h(x_0)<0$ by \eqref{F:at:y0}-\eqref{h:def:F} and \eqref{h:prima:y2} and \eqref{h:prima:0}. The proof is now complete since \eqref{g:prima} is impossible.
\end{proof}

\section{Proof of Theorem \ref{Main:1:Thm1}}
The first statement concerning the existence and nonexistence of solutions of $(P_\lambda)$ follows from Lemma \ref{Uniq:sol} and Lemma \ref{max:sol}.

We show the uniqueness and the multiplicity of solutions of $(P_{\lambda})$ as follows: let $v$ be the unique solution of \eqref{Eq:v:alfa}. Let $t_0 \in \RR$ be fixed. Set $s_0 = e^{t_0}$ and choose a negative constant $A$ depending on $s_0$ such that
\begin{equation}\label{As:lambdaS0}
v(s_0)=-(1-A)^{-1}
\end{equation}
holds.

Now we recall the rescaling
\[
s=\left(\frac{\lambda}{\tilde{\lambda}(k,\sigma)}\right)^{\frac{1}{2k+\sigma}}\left(1-A \right)^{\frac{q-k}{2k+\sigma}}r,\, r>0,
\]
used to obtain \eqref{Eq:v:alfa}. Set $r=1$ and choose a positive constant $\lambda$ depending on $s_0$ such that
\begin{equation}\label{lambdaS0}
s_0=\left(\frac{\lambda}{\tilde{\lambda}(k,\sigma)}\right)^{\frac{1}{2k+\sigma}}\left(1-A \right)^{\frac{q-k}{2k+\sigma}}.
\end{equation}

Next, we write the value $v(s_0)$ in terms of the point $(x(t_0), y(t_0))$ with $s_0=e^{t_0}$ (see \eqref{v:s:t}) as follows
\begin{equation}\label{v:s0:t0}
v(s_0)=-\left[s_{0}^{2k+\sigma}\bar{\lambda}(k,\sigma)\right]^{-\frac{1}{q-k}}(x(t_0)y(t_0)^k)^{\frac{1}{q-k}}.
\end{equation}

Note that $\lambda = c_{n,k}x(t_0)y(t_0)^k$ by \eqref{As:lambdaS0}-\eqref{v:s0:t0}. Further, the map $s_0 \mapsto \lambda$ has range $(0,\tilde{\lambda})$ by the continuity of the orbit, $\lim_{t\to -\infty}y(t) =0$ and $\lim_{t\to +\infty}c_{n,k}x(t)y(t)^k =\tilde{\lambda}$. Now we define
\[
u_{\lambda}(r) = 1 + (1-A)v(s).
\]
Note that $u_{\lambda}(0)=A$ and $u_{\lambda}$ solves problem $(P_\lambda)$. Now we discuss the multiplicity and the uniqueness of solutions. We consider the line $y=\lambda \tau_{\sigma} / \tilde{\lambda}$. Then for each intersection between the line and the orbit, we obtain one and/or several times $t_0$ and then one and/or several $s_0$ as well. Further, if we assume that there exist two different points $t_1, t_2 \in \RR$ such that at line $y(s_1) = y(s_2)$ holds, then following the same argument as before we find two different solutions of problem $(P_\lambda)$ for the same $\lambda$ by \eqref{As:lambdaS0}.

{\bf Proof of (I)}. In this case the orbit describes a spiral starting from $(n+\sigma, 0)$ and ending at $(\hat{x}, \hat{y})$. Thus the line $y=\lambda \tau_{\sigma} / \tilde{\lambda}$ may intersect one or several times to the orbit. On the other hand, the line  $y=\lambda \tau_{\sigma} / \tilde{\lambda}$ intersects the curve $(x(t), y(t))$ an infinite numbers of times if $\lambda=\tilde{\lambda}$ and finite but large numbers of times if $\lambda$ is sufficiently close to but not equal to $\tilde{\lambda}$. Hence, according to the explication above we conclude the proof.

\medbreak

{\bf Proof of (II)}. Here we apply Lemma \ref{lem:qJL:nodo}. In the case $q\geq q_{JL}(k)$, we see that the line $y=\lambda \tau_{\sigma} /\tilde{\lambda}$ intersects the orbit $(x(t),y(t))$ only one time for each time $t_0$, which in turn implies that the problem $(P_{\lambda(s_0)})$ admits a unique solution $u_{\lambda(s_0)}$.  Now, if $\tilde{\lambda} <\lambda^*$ we see that the solution $u_{\lambda(s_0)}$ is a maximal solution by Lemma \ref{max:sol}, which is decreasing as $\lambda(s_0)$ increases. Since $\lambda(s_0) \to \tilde{\lambda}$ as $s_0\to +\infty$, we conclude that $u_{\lambda(s_0)}(0) \to A(\infty) >-\infty$ as $s_0\to +\infty$ since $\tilde{\lambda} <\lambda^*$, which is impossible by \eqref{As:lambdaS0}. Therefore, $\tilde{\lambda} =\lambda^*$. This completes the proof of Theorem \ref{Main:1:Thm1}.


\end{document}